\documentclass{amsart}

\usepackage{graphicx,algorithm,algpseudocode,xcolor,stmaryrd,amsaddr}

\algnewcommand\algorithmicinput{\textbf{Input:}}
\algnewcommand\Input{\item[\algorithmicinput]}
\algnewcommand\algorithmicoutput{\textbf{Output:}}
\algnewcommand\Output{\item[\algorithmicoutput]}
\algnewcommand{\LineIf}[2]{\State \algorithmicif\, #1 \,\algorithmicthen\, #2 \,\algorithmicend\ \algorithmicif}
\algnewcommand{\LineForAll}[2]{\State \algorithmicforall\, #1 \,\algorithmicdo\, #2 \,\algorithmicend\ \algorithmicfor}
\algnewcommand{\Accept}{\textbf{accept}}
\algnewcommand{\Reject}{\textbf{reject}}

\newcommand{\Reg}{\mathrm{RegularElement}}
\newcommand{\Regbf}{\mathbf{RegularElement}}
\newcommand{\GenReg}{\mathrm{GeneralizedRegular}}
\newcommand{\GenRegbf}{\mathbf{GeneralizedRegular}}
\newcommand{\GenInv}{\mathrm{GeneralizedInverse}}
\newcommand{\GenInvbf}{\mathbf{GeneralizedInverse}}
\newcommand{\GenWeakInv}{\mathrm{GeneralizedWeakInverse}}
\newcommand{\GenWeakInvbf}{\mathbf{GeneralizedWeakInverse}}
\newcommand{\LeftId}{\mathrm{LeftIdentities}}
\newcommand{\LeftIdbf}{\mathbf{LeftIdentities}}
\newcommand{\RightId}{\mathrm{RightIdentities}}
\newcommand{\RightIdbf}{\mathbf{RightIdentities}}
\newcommand{\Rtrivial}{\mathrm{\gR\text{-}trivial}}
\newcommand{\Rtrivialbf}{\mathbf{\gR\text{-}trivial}}
\newcommand{\Commutative}{\mathrm{Commutative}}
\newcommand{\Commutativebf}{\mathbf{Commutative}}
\newcommand{\LeftZero}{\mathrm{LeftZero}}
\newcommand{\LeftZerobf}{\mathbf{LeftZero}}
\newcommand{\RightZero}{\mathrm{RightZero}}
\newcommand{\RightZerobf}{\mathbf{RightZero}}
\newcommand{\Zero}{\mathrm{Zero}}
\newcommand{\Zerobf}{\mathbf{Zero}}
\newcommand{\CompReg}{\mathrm{CompletelyRegular}}
\newcommand{\CompRegbf}{\mathbf{CompletelyRegular}}
\newcommand{\NilSgp}{\mathrm{NilpotentSemigroup}}
\newcommand{\NilSgpbf}{\mathbf{NilpotentSemigroup}}
\newcommand{\Model}{\mathrm{Model}}
\newcommand{\Modelbf}{\mathbf{Model}}
\newcommand{\DFAIsect}{\mathrm{DFAIntersection}}
\newcommand{\DFAIsectbf}{\mathbf{DFAIntersection}}
\newcommand{\DFAEmp}{\mathrm{DFAEmptiness}}
\newcommand{\DFAEmpbf}{\mathbf{DFAEmptiness}}

\newcommand{\Fix}{\mathrm{Fix}}

\newcommand{\autA}{\mathcal{A}}
\newcommand{\gR}{\mathcal{R}}

\newcommand{\FO}{\ensuremath\mathsf{FO}}
\newcommand{\AC}{\ensuremath\mathsf{AC}}
\newcommand{\PSPACE}{\ensuremath\mathsf{PSPACE}}
\newcommand{\NPSPACE}{\ensuremath\mathsf{NPSPACE}}

\newcommand{\NL}{\ensuremath\mathsf{NL}}
\newcommand{\NP}{\ensuremath\mathsf{NP}}
\newcommand{\co}{\ensuremath\mathsf{co}}
\newcommand{\eq}{\ensuremath\mathrel{=}}

\newcommand{\ie}{i.e.,~}
\newcommand{\set}[2]{\left\{#1\mathrel{\left|\vphantom{#1}\vphantom{#2}\right.}#2\right\}}
\newcommand{\os}[1]{\left\{\mathinner{#1}\right\}}
\newcommand{\abs}[1]{\left|\mathinner{#1}\right|}

\newtheorem{theorem}{Theorem}[section]
\newtheorem{lemma}[theorem]{Lemma}
\newtheorem{corollary}[theorem]{Corollary}

\theoremstyle{remark}
\newtheorem{remark}[theorem]{Remark}

\numberwithin{equation}{section}
\begin{document}

\title{On the Complexity of Properties of Transformation Semigroups}
\date{\today}

\author{Lukas Fleischer}
\address{University of Stuttgart, FMI \\
Universit\"atsstra{\ss}e 38, 70569 Stuttgart, Germany\\[3mm]
School of Computer Science, University of Waterloo \\
200 University Avenue West, Waterloo, ON N2L 3G1, Canada
}
\email{fleischer@fmi.uni-stuttgart.de}

\author{Trevor Jack}
\address{University of Colorado, Department of Mathematics \\ 2300 Colorado Avenue, Boulder, USA}
\email{trevor.jack@colorado.edu}

\thanks{This material is based upon work supported by the National Science Foundation under Grant No. DMS 1500254.}

\begin{abstract}
  We investigate the computational complexity for determining various properties of a finite transformation semigroup given by generators.
  We introduce a simple framework to describe transformation semigroup properties that are decidable in $\AC^0$. This framework is then used to show that the problems of deciding whether a transformation semigroup is a group, commutative or a semilattice are in $\AC^0$.
  Deciding whether a semigroup has a left (resp.~right) zero is shown to be $\NL$-complete, as are the problems of testing whether a transformation semigroup is nilpotent, $\gR$-trivial or has central idempotents.
  We also give $\NL$ algorithms for testing whether a transformation semigroup is idempotent, orthodox, completely regular, Clifford or has commuting idempotents.
  Some of these algorithms are direct consequences of the more general result that arbitrary fixed semigroup equations can be tested in~$\NL$.
  Moreover, we show how to compute left and right identities of a transformation semigroup in polynomial time.
  Finally, we show that checking whether an element is regular is $\PSPACE$-complete.
\end{abstract}

\maketitle

\section{Introduction}
\label{sec:intro}

Given a permutation group by generators, many of its properties, like size and membership, can be determined in polynomial time using Sims' stabilizer chains.
In contrast, the known algorithms for the corresponding problems for transformation semigroups given by generators often rely on an enumeration of the $\gR$-classes of the semigroup, which already requires exponential time \cite{EA:CFS,MI:GAP}. Moreover, the membership problem for transformation semigroups is known to be $\PSPACE$-complete \cite{KO:LBN}, as is checking whether two elements are $\gR$-related \cite{BS:CATM} and whether a transformation semigroup is aperiodic \cite{CH:FAA}.
In this work, we will show that many similar problems, such as checking whether a specific element is regular and variants thereof, are also $\PSPACE$-complete (Theorem~\ref{RegThm}).

While the hardness results above indicate that problems concerning Green's relations, inverses and idempotent elements are generally hard in the transformation semigroup setting, this is certainly not true in general. Clearly, the problem of checking whether a given transformation is idempotent is in $\AC^0$, the class of all sets decidable by unbounded fan-in Boolean circuits of constant depth: it suffices to verify that the transformation is the identity on its image.
Similarly, we show that the problems of testing whether a transformation semigroup is a group, commutative or a semilattice are in $\AC^0$.
Our results even hold in the \emph{uniform} setting where so-called \emph{direct connection languages} of the circuits are required to be decidable in logarithmic time using a deterministic random-access Turing machine.

Aside from that, we prove that many decision problems for transformation semigroups are complete for $\NL$, the class of all sets decidable in non-deterministic logarithmic space. This includes testing whether a transformation semigroup
\begin{itemize}
\item contains a left zero, a right zero or a zero (Theorem~\ref{thm:zero});
\item is nilpotent (Theorem~\ref{thm:nil});
\item is $\gR$-trivial (Theorem~\ref{thm:Rtrivial});
\item has central idempotents (Theorem~\ref{CentralIdempotentsThm}).
\end{itemize}

For some problems, we establish membership in $\NL$, but we leave hardness results as open problems. These include testing whether a transformation semigroup
\begin{itemize}
\item is a band (Corollary~\ref{IdemCor});
\item has commuting idempotents (Corollary~\ref{IdemCor});
\item is orthodox (Corollary~\ref{IdemCor});
\item is completely regular (Theorem~\ref{ComRegThm});
\item is a Clifford semigroup (Corollary~\ref{CliffCor}).
\end{itemize}
We also describe an $\NL$ algorithm to check whether a transformation semigroup models a fixed identity or a fixed quasi-identity involving idempotents (Theorem~\ref{thm:model}) and prove that this problem is $\NL$-complete for very simple identities already (Theorem~\ref{SimpleThm}).

Additionally, we show that the left and right identities of a transformation semigroup can be enumerated in polynomial time (Theorems~\ref{LidsThm} and~\ref{RidsThm}).

\section{Preliminaries} \label{NotationSection}

Let $S$ be a semigroup.
An element $\ell$ of a semigroup $S$ is a \emph{left identity} if $\ell s = s$ for all $s \in S$. 
An element $r$ of a semigroup $S$ is a \emph{right identity} if $sr = s$ for all $s \in S$. 
An element $\ell$ of a semigroup $S$ is a \emph{left zero} if $\ell s=\ell$ for all $s \in S$.
An element $r$ of a semigroup $S$ is a \emph{right zero} if $sr=r$ for all $s \in S$.
An element that is both a left and a right zero is called \emph{zero}. If a semigroup contains a zero element, this element is unique and usually denoted by $0$.
If a semigroup contains a left zero $\ell$ and a right zero $r$, then $\ell = \ell r = r$ is a zero. 
A semigroup $S$ that has a zero, $0 \in S$, is called \emph{d-nilpotent} if $S^d = \{0\}$ for a fixed $d \in \mathbb{N}$. We call $S$ \emph{nilpotent} if it is $d$-nilpotent for some $d \in \mathbb{N}$. The smallest such $d$ is called its \emph{nilpotency degree}. 

An element $e \in S$ is \emph{idempotent} if $e^2 = e$.
A semigroup is \emph{idempotent} if all elements are idempotent. Following classical terminology, idempotent semigroups are often also referred to as \emph{bands}.
A semigroup is \emph{orthodox} if the product of any two idempotents is again idempotent.
It is well-known that in every finite semigroup $S$, there is a natural number $\omega_S$ such that $s^{\omega_S}$ is idempotent for every $s \in S$. If the reference to $S$ is clear from the context, we usually write $\omega$ instead of $\omega_S$.

Two elements $s, t \in S$ \emph{commute} if $st = ts$. An element $s \in S$ is \emph{central} if it commutes with every other element of $S$. A semigroup $S$ is \emph{commutative} if all pairs of elements of $S$ commute. A semigroup that is both idempotent and commutative is also called a \emph{semilattice}.

An element $s \in S$ is \emph{regular} if there exists $t \in S$ such that $sts = s$. An element $t \in S$ such that $tst = t$ is called a \emph{weak inverse} of $s$. If $sts = s$ and $tst = t$, then $t$ is an \emph{inverse} of $s$. 
A semigroup is called \emph{regular} if all of its elements are regular. 
A semigroup is \emph{inverse} if every element has a unique inverse.
A semigroup $S$ is \emph{completely regular} if every element of $S$ belongs to some subgroup of $S$.
A semigroup is a \emph{Clifford semigroup} if it is completely regular and its idempotents commute. 

A semigroup is called \emph{aperiodic} if it does not contain any nontrivial subsemigroup that is a group.
A semigroup $S$ is \emph{$\gR$-trivial} if it does not contain two distinct elements $s, t \in S$ such that $sS \cup \os{s} = tS \cup \os{t}$.

A \emph{variety of finite semigroups} is a class of finite semigroups that is closed under finite direct products and under taking divisors. In the literature, such classes of semigroups are often also referred to as pseudovarieties.

The \emph{full transformation semigroup} over some set $Q$ is the set of all mappings $f \colon Q \to Q$, the so-called \emph{transformations over~$Q$}, together with function composition.
Subsemigroups of the full transformation semigroup are often also referred to as \emph{transformation semigroups}.
The elements of $Q$ are sometimes referred to as \emph{points}.

For $n \in \mathbb{N}$, we define $[n] := \{1,\dots,n\}$ and we use $T_n$ to denote the full transformation semigroup over the set $[n]$. For elements $a_1,\dots,a_k \in T_n$, let $\langle a_1,\dots,a_k \rangle$ be the subsemigroup of $T_n$ generated by $a_1,\dots,a_k$.

For $s \in T_n$ and $A \subseteq T_n$, define the \emph{kernel of $s$} (resp.~$A$) as
\begin{align*}
\ker(s) & := \set{(p,q) \in [n]^2}{ps = qs}, \\
\ker(A) & := \bigcap_{a \in A} \ker(a).
\end{align*}
For $S = \langle a_1,\dots,a_k \rangle$, note that $\ker(S) = \ker(\{a_1,\dots,a_k\})$ and that $\ker(S)$ forms an equivalence relation on $[n]$. Then $[n] / \ker(S)$ is a partition into equivalence classes, $\llbracket q \rrbracket := \set{p \in [n]}{ps = qs \text{ for every } s \in S}$. We can define a natural action of $S$ on these classes by the following homomorphism:
\[ S \rightarrow T_{[n] / \ker(S)}, s \mapsto \overline{s}\]
with
\[ \llbracket q \rrbracket \overline{s} = \llbracket qs \rrbracket \text{ for } q \in [n].\]
Let $\overline{S} := \{\overline{s}: s \in S\}$.

Denote the image of a semigroup element $s \in S$ as
\[[n]s := \set{ q \in [n]}{ q = ps \text{ for some } p \in [n] \text{ and some } s \in S} \]

Define $[n]S := \bigcup_{s \in S} [n]s$. We can define a natural action of $S$ on its image by the following homomorphism:
\[ S \rightarrow T_{[n]S}, s \mapsto \widetilde{s}\]
with
\[ q \widetilde{s} =  qs \text{ for } q \in [n]S\]
and
\[\widetilde{S} := \{\widetilde{s}: s \in S\}.\]

Let $S$ be a transformation semigroup, \ie $S := \langle a_1,\dots,a_k \rangle \leq T_n$ for some $n,k \in \mathbb{N}$. For $d \in \mathbb{N}$, we define $S$ to act upon $d$-tuples $[n]^d$ component-wise: for $(q_1,\dots,q_d) \in [n]^d$ and $s \in S$,
\[(q_1,\dots,q_d)s = (q_1s,\dots,q_ds).\]

For a set of transformations $A:=\{a_1,\dots,a_k\}$ acting on a set $Q$, define the \emph{transformation graph} $\Gamma(A,Q)$ as having vertices $Q$ and directed edges
\[E := \set{(p,q) \in Q^2}{\exists i \in [k](pa_i=q)}.\]

Denote the \emph{pre-image} of an element $q \in Q$ as 
\[S^{-1}(q) := \set{q \in Q}{ps=q \text{ for some } s \in S}.\]

We use the notation
\[\Fix(A,Q) := \set{ q\in Q}{qa=q \text{ for all } a\in A }\]
to denote the set of \emph{fixed points of $A$}.
    
\section{First-Order Definable Properties}

In this section, we will investigate classes of transformation semigroups with very efficient membership tests. To this end, we introduce a variant of first-order logic that is used to define classes of semigroups. We allow quantification over the set of generators $A$ and over points of the underlying set $Q$. The only allowed predicates are of the form $p \cdot a_1 \cdots a_k = q \cdot b_1 \cdots b_\ell$ where $p, q$ are variables corresponding to points and $a_1, \dots, a_k, b_1, \dots, b_\ell$ are variables corresponding to generators. For example, the formula
\begin{align*}
  \forall a \in A \, \forall p, q \in Q \colon p \cdot a = q \cdot a
\end{align*}
can be used to express that the image of every element of $S$ is a singleton.
It follows immediately from the well-known result $\FO = \AC^0$ \cite[Theorem 5.22]{IM:DC} that the membership problem for classes of semigroups defined by such formulas is in $\AC^0$.

\begin{remark}
  It is important to note that the logical formalism described above does not allow quantification over elements of the generated transformation semigroup~$S$.
  In fact, allowing such quantifiers yields a much more expressive formalism. For example, the following formula can be used to define the class of all aperiodic transformation semigroups:
  \begin{align*}
      \forall s,t \in S \colon (s \mathrel{\mathcal{H}} t \to s = t).
  \end{align*}
  Here, $s \mathrel{\mathcal{H}} t$ is a shorthand for the formula
  \begin{align*}
      (s = t) \lor \big(& (\exists r \in S: sr = t) \land (\exists r \in S: rs = t) \land {}\\
                        & (\exists r \in S: tr = s) \land (\exists r \in S: rt = s)\big),
  \end{align*}
  $s = t$ is a short form for $\forall q \in Q \colon q \cdot s = q \cdot t$, and formulas of the form $rs = t$ are short forms for $\forall q \in Q \colon q \cdot rs = q \cdot t$. Testing whether a transformation semigroup is aperiodic has been shown to be $\PSPACE$-complete in~\cite{CH:FAA}.
\end{remark}

The class of commutative transformation semigroups is easily described using this formalism.
Therefore, we can design an efficient algorithm for the following decision problem:

\medskip
$\Commutativebf$
\begin{itemize}
\item Input: $a_1,\dots,a_k \in T_n$
\item Problem: Is $\langle a_1,\dots,a_k \rangle$ commutative?
\end{itemize}

\begin{theorem}
  $\Commutative$ is in $\AC^0$.
\end{theorem}
\begin{proof}
  Clearly, a semigroup is commutative if and only if all generators commute. Therefore, commutativity can be expressed by the formula
  \begin{align*}
    \forall a, b \in A \, \forall q \in Q \colon q \cdot ab = q \cdot ba.
  \end{align*}
  Thus, testing whether a transformation semigroup is commutative is in $\AC^0$.
\end{proof}

\begin{theorem}
The problem of determining if a transformation semigroup is a semilattice is in $\AC^0$.
\end{theorem}
\begin{proof}
  We can use the formula
  \begin{align*}
    \forall a \in A \, \forall q \in Q \colon q \cdot a^2 = q \cdot a
  \end{align*}
  to express that every generator is idempotent. Together with the previous theorem, this yields the desired statement.
\end{proof}

Our formula for permutation groups will be based on the following lemma.

\begin{lemma}\label{lem:groups}
  Let $a_1, \dots, a_k \in T_n$. Then, $S: = \langle a_1, \dots, a_k \rangle$ is a group if and only if the following three properties hold:
  \begin{enumerate}
      \item all generators have the same image $X \subseteq Q$,
      \item all generators are permutations on $X$, and
      \item all generators have the same kernel $\ker(a_i)$.
  \end{enumerate}
\end{lemma}
\begin{proof}
  Suppose that $S$ is a group.
  If one of the generators $a_i$ is not a permuation on its image, then the image of $a_i^2$ is strictly smaller than the image of $a_i$, which yields $a_i^{\omega+1} \ne a_i$. Since in a group, the only idempotent element is the identity element, this is a contradiction. We may therefore assume that all generators are permutations on their images.
  Therefore, if two generators $a_i$ and $a_j$ do not have the same image, then the same holds for $a_i^\omega$ and $a_j^\omega$, contradicting the fact that a group contains a unique idempotent element.
  If there exist $p, q \in [n]$ and $i, j \in [k]$ such that $pa_i = qa_i$ but $pa_j \ne qa_j$, we have $pa_i^\omega = qa_i^\omega$ and $pa_j^\omega \ne qa_j^\omega$ using the observation that $a_j$ is a bijection on its image. This again yields two different idempotent elements.
  
  Conversely, suppose that the three properties stated above hold. By induction, all three properties then also hold for every element of $S$. Let $s \in S$. We claim that $s^\omega$ is the identity element which suffices to conclude the proof.
  To this end, we show that $s^\omega a_i = a_i = a_i s^\omega$ for all $i \in [k]$.
  Let $q \in [n]$. Clearly, since $s$ is a bijection on the image of $a_i$, the transformation $s^\omega$ is the identity on $qa_i$, thus $qa_is^\omega = qa_i$.
  Let $p = qs^\omega$. Clearly, $ps^\omega = qs^\omega$, \ie $(p, q) \in \ker(s^\omega)$. Since all elements have the same kernel, this yields $pa_i = qa_i$ and thus, $qs^\omega a_i = q a_i$.
\end{proof}

\begin{theorem}
  The problem of determining whether a transformation semigroup is a group is in $\AC^0$.
\end{theorem}
\begin{proof}
  We provide formulas for all the conditions stated in Lemma~\ref{lem:groups}.
  The first condition can be expressed by the formula
  \begin{align*}
    \forall a, b \in A \, \forall  q \in Q \colon (\exists p \in Q \colon pa = q) \,\to\, (\exists p \in Q \colon pb = q).
  \end{align*}
  The second property is expressed by
  \begin{align*}
    \forall a \in A \, \forall p, q \in Q \colon pa \ne qa \rightarrow pa^2 \ne qa^2.
  \end{align*}
  The property of matching kernels is defined by the
  \begin{align*}
    \forall a, b \in A \, \forall p, q \in Q \colon (pa = qa \,\leftrightarrow\, pb = qb).
  \end{align*}
  This suffices to show that the problem of deciding whether a transformation semigroup is a group is in $\AC^0$.
\end{proof}

\section{Locally Testable Properties} \label{LocalSection}

In this section, we investigate properties of a transformation semigroup $S$ with an underlying set $[n]$ that can be checked simply by verifying certain conditions on each vertex of the graph $\Gamma(S,[n])$.
We first investigate the problem of testing whether a transformation semigroup has a right zero.

\medskip
$\RightZerobf$
\begin{itemize}
\item Input: $a_1,\dots,a_k \in T_n$
\item Problem: Does $\langle a_1,\dots,a_k \rangle$ have a right zero?
\end{itemize}

\begin{lemma}\label{Right0} Let $k,n \in \mathbb{N}$ and let $a_1,\dots,a_k \in T_n$. An element $r$ is a right zero of $S := \langle a_1,\dots,a_k\rangle$ iff $[n] = \bigcup_{p \in [n]r} S^{-1}(p)$ is a disjoint union. \end{lemma}
\begin{proof}
For the forward direction, assume $r$ is a right zero. Note that $r^2 = r$. Hence, $r$ fixes its images.
Pick any $q \in [n]$ and $s,t \in S$ such that $qs,qt \in [n]r$. Then $qs = qsr = qr = qtr = qt$ and thus $[n] = \bigcup_{p \in [n]r} S^{-1}(p)$ is a disjoint union.

Conversely, assume $S$ has an element, $r$, such that $[n] = \bigcup_{p \in [n]r} S^{-1}(p)$ is disjoint.
Pick any $q \in [n]$ and any $s \in S$. Then $q \in S^{-1}(qr)$ and $q \in S^{-1}(qsr)$. But since the union is disjoint, this means $qr = qsr$. Thus, $r$ is a right zero.
\end{proof}

\begin{lemma}\label{Right0Graph}
  Let $k,n \in \mathbb{N}$ and let $a_1,\dots,a_k \in T_n$. Then $S := \langle a_1,\dots,a_k\rangle$ has a right zero iff for every pair $p,q \in [n]$ that are in the same connected component of $\Gamma(\{a_1,\dots,a_k\},[n])$, there is some $s \in S$ such that $ps = qs$.
\end{lemma}
\begin{proof}
  For the forward direction, assume $r$ is a right zero of $S$. Then, by Lemma \ref{Right0}, $[n] = \bigcup_{q' \in [n]r} S^{-1}(q')$ is a disjoint union. Thus, the $S^{-1}(q')$ sets are the vertex sets of the connected components of $\Gamma(\{a_1,\dots,a_k\},[n])$. Moreover, for each $p,q \in S^{-1}(q')$, $pr = q' = qr$.

  For the converse, let $w \in S$ have minimal image size. Pick any $p,q \in [n]$ in the same connected component $V$.
  Note that $pw,qw \in V$. By assumption, then, there is some $s \in S$ such that $pws = qws$. Since $|[n]w|$ is minimal, then $|[n]w| = |[n]ws|$ and hence $pw = qw$. That is, $|Vw|=1$ for each connected component $V$. Thus, $[n] = \bigcup_{q' \in [n]w}S^{-1}(q')$ is disjoint. Therefore, by Lemma \ref{Right0}, $w$ is a right zero.
\end{proof}

\begin{lemma}\label{lem:right-zero-nl}$\RightZero$ is in $\NL$.\end{lemma}
\begin{proof}
  By Lemma \ref{Right0Graph}, we need only show every pair of vertices from the same connected component can be collapsed by some semigroup element to a common image. First, compute the edges of $\Gamma(\{a_1,\dots,a_k\},[n])$ with a log-space transducer. Then for each pair $(p,q) \in [n]$, use a log-space algorithm for undirected graph connectivity \cite{RE:UCL} to determine if $p$ and $q$ are in the same connected component. If not, proceed to the next pair. If so, guess generators for an $s \in S$ such that $ps = qs$.
\end{proof}

Similarly, one can ask whether a semigroup has a left zero.
Due to lack of left-right symmetry in transformation semigroups, this requires a different approach.

\medskip
$\LeftZerobf$
\begin{itemize}
\item Input: $a_1,\dots,a_k \in T_n$
\item Problem: Does $\langle a_1,\dots,a_k \rangle$ have a left zero?
\end{itemize}

\begin{lemma}\label{Left0} Let $k,n \in \mathbb{N}$. Let $a_1,\dots,a_k \in T_n$ and $S := \langle a_1,\dots,a_k\rangle$. Then, $S$ contains a left zero iff for each $q\in [n]$, there exists $s \in S$ such that $qs \in \Fix(S,[n])$. \end{lemma}
\begin{proof}
For the forward direction, assume $\ell \in S$ is a left zero. Then for any $q \in [n]$ and any $s \in S$, $(q \ell)s = q\ell$. Thus, $q\ell \in \Fix(S,[n])$.

Conversely, assume that for any $q \in [n]$, there exists $s \in S$ such that $qs \in \Fix(S,[n])$. Let $s \in S$ be an element such that $[n]s$ is of minimal size. Note that $\Fix(S,[n]) \subseteq [n]s$. Assume for contradiction there is $q \in [n]s \setminus \Fix(S,[n])$. We know there is an $t \in S$ such that $qt \in \Fix(S,[n])$. Since $\Fix(S,[n])t = \Fix(S,[n])$, then $|[n]st| < |[n]s|$, a contradiction. Therefore, $[n]s = \Fix(S,[n])$. Consequently, $(qs)t = qs$ for every $t \in S$ and thus $s$ is a left zero.
\end{proof}

\begin{lemma}\label{lem:left-zero-nl} $\LeftZero$ is in $\NL$. \end{lemma}
\begin{proof}
By Lemma \ref{Left0}, we need only verify that for each $q \in [n]$, there is an $s \in S$ such that $qs \in \Fix(S,[n])$. Iterate through $q \in [n]$, initializing $p := q$. Then, repeatedly guess generators $a_i \in \{a_1,\dots,a_k\}$ and let $p = p a_i$. Non-deterministically stop guessing generators and check that $p \in \Fix(S,[n])$ by iterating through $i \in [k]$ and verifying that $pa_i = p$.
\end{proof}

We also consider the problem of deciding whether a semigroup contains a zero.

\medskip
$\Zerobf$
\begin{itemize}
\item Input: $a_1,\dots,a_k \in T_n$
\item Problem: Does $\langle a_1,\dots,a_k \rangle$ have a zero?
\end{itemize}

We now prove that $\LeftZero$, $\RightZero$, and $\Zero$ are all $\NL$-hard by a reduction of the following problem known to be $\NL$-complete \cite[Theorem 26]{JO:SB}.

\medskip
$\DFAEmpbf$
\begin{itemize}
\item Input: A deterministic finite automaton (DFA) $\autA$ over an alphabet $\Sigma$.
\item Problem: Is there $w \in \Sigma^*$ that is accepted by $\autA$?
\end{itemize}

\begin{theorem} \label{thm:zero}
$\LeftZero$, $\RightZero$, and $\Zero$ are $\NL$-complete.
\end{theorem}
\begin{proof}
A semigroup contains a zero element iff it contains both a left zero and a right zero. Then $\Zero$ is in $\NL$ as a corollary to Lemmas~\ref{lem:left-zero-nl} and \ref{lem:right-zero-nl}. It remains to prove that all three problems are $\NL$-hard.

Given a DFA with state set $[n]$, initial state $q_o \in [n]$, accepting states $F \subseteq [n]$, and transformations $a_1,\dots,a_k \in T_n$, we construct a subsemigroup of $T_{n+1}$ as follows. First extend $a_1,\dots,a_k$ by defining them to fix $n+1$. Then define the following new transformations:

$$xb := \begin{cases}
q_0 & \text{if }q \in [n],\\
n+1 & \text{if }q = n+1.
\end{cases}
\hspace{1cm}
xc := \begin{cases}
q & \text{if }q \in [n+1] \setminus F,\\
n+1 & \text{if }q \in F \text{ or } q = n+1.
\end{cases}
$$

We claim $S = \langle a_1,\dots,a_k,b,c \rangle$ has a zero iff the language of the DFA is nonempty; that is, there exists $w \in \{a_1,\dots,a_k\}^*$ such that $q_0w \in F$.

Assume there exists $w \in \{a_1,\dots,a_k\}^*$ such that $q_0w \in F$. Then $[n+1]bwc = \{n+1\}$ and thus $S$ has a zero element. Conversely, assume $S$ has a zero element $0 \in S$. Since $n+1$ is the only point fixed by all transformations and $(q_00)s = q_00$ for every $s \in S$, then $q_00 = n+1$. The only transformation that sends something from $[n]$ to $n+1$ is $c$, which only sends elements of $F$ to $n+1$. So, $0 = w_1cw_2$ where $w_1,w_2 \in S$, $q_0w_1 \in F$, and $w_1 \in {a_1,\dots,a_k,b}^*$. Because $[n]b = \{q_0\}$, then either (case 1) $w_1 \in \{a_1,\dots,a_k\}^*$ or (case 2) $w_1 = w_3bw_4$ where $w_4 \in \{a_1,\dots,a_k\}^*$ and $q_0w_4 \in F$. In either case, we have a word from $\{a_1,\dots,a_k\}^*$ that is accepted by the DFA.

Finally, note that $\Gamma(\{a_1,\dots,a_k,b,c\},[n+1])$ has a single connected component. By Lemma \ref{Right0}, any right zero $r \in S$ must have an image size of a single point. Since $\Fix(S,[n+1]) = \{n+1\}$, the only right zero of $S$ would be the zero element, if it exists. Similarly, since the images of left zeroes are fixed points, the only left zero of $S$ would be the zero element, if it exists.
\end{proof}

We define the following problem and show that it is in $\NL$:

\medskip
$\NilSgpbf$
\begin{itemize}
\item Input: $a_1,\dots,a_k \in T_n$
\item Problem: Is $\langle a_1,\dots,a_k \rangle$ nilpotent?
\end{itemize}

\begin{lemma}\label{NilLemma} Let $S:=\langle a_1,\dots,a_k\rangle \leq T_n$. The following properties are equivalent:
\begin{enumerate}
\item $S$ is nilpotent;
\item $S$ has a zero element, $0$, and $\Gamma(\{a_1,\dots,a_k\},[n] \setminus [n]0)$ is acyclic;
\item $S$ is $n$-nilpotent.
\end{enumerate}
\end{lemma}
\begin{proof}
We prove $(1) \Rightarrow (2)$ by contraposition.
Certainly, if $S$ does not have a zero element, it is not nilpotent by definition. Now assume there is a zero element, $0$, but also a cycle in $\Gamma(\{a_1,\dots,a_k\},[n] \setminus [n]0)$.
Let the cycle be $(q_1,\dots,q_i,q_1)$. Then there are generators $b_1,\dots,b_i \in \{a_1,\dots,a_k\}$ such that $q_jb_j = q_{j+1}$ for $1 \leq j < i$ and $q_ib_i = q_1$.
Let $s = b_1 \cdot \dots\cdot b_i$ and note that $q_1s = q_1$ and that $q_1 \not \in [n]0$. So, $q_1s^d = q_1 \neq q_10$ for any $d \in \mathbb{N}$. 

For $(2) \Rightarrow (3)$, assume $S$ has a zero element, 0, and assume that the graph $\Gamma(\{a_1,\dots,a_k\}, [n]\setminus [n]0)$ is acyclic.
By Lemma \ref{Right0}, the connected components of $G := \Gamma(\{a_1,\dots,a_k\}, [n])$ are formed by the vertex sets $S^{-1}(q)$ for $q\in[n]0$. Hence the only cycles in $G$ are the loops around each vertex in $[n]0$. Let $d$ be the diameter of $G$; that is, the length of the longest path from any $p\in [n]$ to $p0$. Clearly $d\leq n$. We claim that $S$ is nilpotent of degree $d$, and hence $n$-nilpotent.
Pick any $b_1,\dots,b_d\in\{a_1,\dots,a_k\}$. For each $1\leq i \leq d$, let $t_i := b_1\dots b_i$. For each $p\in [n]$ the directed path $(p,pt_1,\dots,pt_d)$ has $d$ vertices. Since $d$ is the diameter, the final vertex must be $p0$. Thus $t_d=0$ and $S^d=\{0\}$.

Certainly, $(3) \Rightarrow (1)$.
\end{proof}

We are now able to describe an $\NL$ algorithm for testing nilpotency.

\begin{lemma} \label{lem:nil} $\NilSgp$ is in $\NL$. \end{lemma}
\begin{proof}
  By Lemma \ref{NilLemma}, $S$ is nilpotent if and only if $x_1 \cdots x_n = x_{n+1} \cdots x_{2n+1}$ for all $x_1, \dots, x_{2n+1} \in S$. Since $\NL$ is closed under complementation, it suffices to show that we can decide in non-deterministic logarithmic space whether there exist $x_1, \dots, x_{2n+1} \in S$ with $x_1 \cdots x_n \ne x_{n+1} \cdots x_{2n+1}$.

  To this end, we first guess an integer $q \in [n]$ and let $p := q$. Then, in a loop, we guess an integer $i \in [k]$ and let $q := q a_i$. This loop is repeated $n$ times in total. In another loop that is iterated $n+1$ times, we repeatedly guess an integer $i \in [k]$ and let $p := q a_i$. Finally, we verify that $p \ne q$.
\end{proof}

To show that $\NilSgp$ is $\NL$-complete, we will use the following lemma. The statement is slightly more general than needed. This allows us to reuse the lemma in the next section.

\begin{lemma}\label{lem:nil-hard}
  There exists a log-space transducer which, given a DFA $\autA$, produces a transformation semigroup S given by generators with the following properties:
  \begin{enumerate}
      \item If $L(\autA) \ne \emptyset$, then there exists an idempotent in S that is not a left zero.
      \item If $L(\autA) = \emptyset$, then every square in S is zero.
  \end{enumerate}
\end{lemma}
\begin{proof}
  Suppose the input consists of a DFA with state set $[n]$, initial state $q_0
  \in [n]$, accepting states $F \subseteq [n]$ and transformations $a_1, \dots,
  a_k \in T_n$. We assume without loss of generality that $q_0 \not\in F$.
  We define $n^2$ transformations $a_{i, j}$ with $i, j \in [n]$ on $[n] \times
  [n] \cup \os{0}$ by setting $0 a_{i,j} = 0$ and
  \begin{align*}
    (q, k) a_{i, j} := \begin{cases}
      (q a_i, k+1) & \text{if $j = k < n$ and $q \not\in F$}, \\
      0 & \text{otherwise}
    \end{cases}
  \end{align*}
  for all $q, k, i, j \in [n]$. We define an additional transformation $b$ by
  $0b = 0$ and
  \begin{align*}
    (q, k) b := \begin{cases}
      (q_0, 1) & \text{if $q \in F$}, \\
      0 & \text{otherwise}
    \end{cases}
  \end{align*}
  for all $q, k \in [n]$. We now show that $S = \langle a_{i,j}, b \mid i, j \in
  [n] \rangle$ satisfies the two properties stated in the lemma.

  Suppose that the language accepted by the DFA is non-empty, \ie there exists
  a word $u \in \os{a_1, \dots, a_k}^*$ with $q_0 u \in F$. Without loss of
  generality, we may assume that $\abs{u} < n$, \ie $u = a_{i_1} \dots
  a_{i_\ell}$ for some $\ell \in [n-1]$ and $i_1, \dots, i_\ell \in [k]$. Let
  $v = a_{i_1, 1} \dots a_{i_\ell, \ell} b$. Note that $b^2$ is
  the transformation that maps every element to~$0$.
  It is easy to verify that $v^2 = v$.
  By construction, we have
  $(q_0, 1) v = (q_0,1)$ and $(q_0, 1) v b^2 = 0$. Thus, $v b^2 \ne v$ and $v$ is not a left zero.

  Conversely, it is easy to see that if the language accepted by the DFA is
  empty, then $s^2 = b^2$ for every $s \in S$. As mentioned before, $b^2$ is a zero element.
\end{proof}

We conclude by proving that testing nilpotency is $\NL$-complete.

\begin{theorem}\label{thm:nil}
  $\NilSgp$ is $\NL$-complete.
\end{theorem}
\begin{proof}
  Lemma~\ref{lem:nil} shows that $\NilSgp$ belongs to $\NL$.
  Lemma~\ref{lem:nil-hard} immediately yields a reduction from $\DFAEmp$ to $\NilSgp$. If case (1) from Lemma~\ref{lem:nil-hard} holds, then let $e$ be the idempotent that is not a left zero and note that $e^d \neq 0$ for any $d \in \mathbb{N}$. Thus, the semigroup is not nilpotent. If case (2) holds, then the only idempotent element is the zero element. Since every element has an idempotent power and $S$ has finitely many elements, then $S$ is nilpotent.
\end{proof}

We will now show that the following problem is NL-complete. 

\medskip
$\Rtrivialbf$
\begin{itemize}
\item Input: $a_1,\dots,a_k \in T_n$
\item Problem: Is $\langle a_1,\dots,a_k \rangle$ $\gR$-trivial?
\end{itemize}

\begin{lemma}\label{lem:Rtrivial} Let $a_1,\dots,a_k \in T_n$. Then, $S := \langle a_1,\dots,a_k\rangle$ is $\gR$-trivial if and only if each cycle in  $\Gamma(\{a_1,\dots,a_k\},[n])$ is of length 1.\end{lemma}
\begin{proof}
  For the implication from left to right, suppose that $\Gamma(\{a_1,\dots,a_k\},[n])$ contains a cycle longer than length 1. Then, there exist $q \in Q$ and $i_1, ..., i_\ell \in [k]$ such that $q \cdot a_{i_1} \ne q$ and $q \cdot a_{i_1} \cdots a_{i_\ell} = q$.
  Let $e := (a_{i_1} \cdots a_{i_\ell})^\omega$. Since $q \cdot e = q$ and $q \cdot ea_{i_1} \ne q$, the elements $e$ and $ea_{i_1}$ are distinct. But $ea_{i_1}a_{i_2}\cdots a_{i_\ell} (a_{i_1}\cdots a_{i_\ell})^{\omega-1} = e^2 = e$, thus $eS^1 = ea_{i_1}S^1$.
  
  For the converse direction, suppose $sS^1 = tS^1$ and $s \neq t$. Then there exists $u,v \in S^1$ and $q \in [n]$ such that $su = t$, $tv=s$, and $qs \neq qt$. Then $qs = qsuv$ and $qs \neq qt = qsu$, yielding a cycle longer than length 1.
\end{proof}

\begin{lemma}\label{lem:Rtrivial-hard}
  There exists a log-space transducer which, given a directed graph $G = (V, E)$, produces a transformation semigroup given by generators with the following properties:
  \begin{enumerate}
      \item If the graph is acyclic, then the semigroup is nilpotent.
      \item If the graph is not acyclic, then the semigroup is not $\gR$-trivial and there exists an idempotent that is not central.
  \end{enumerate}
\end{lemma}
\begin{proof}
  The underlying set is $Q = V \cup \os{0}$ and we add one generator $a_{(v, w)}$ for each edge $(v, w) \in E$ which we define by $v a_{(v, w)} = w$ and $u a_{(v, w)} = 0$ for all $u \ne v$.
  
  Clearly, if $G$ is acyclic, any product of more than $\abs{V}$ generators is the zero element, \ie the transformation that maps each point to~$0$. In this case, the semigroup is nilpotent.
  
  Conversely, if $G$ contains a cycle $v_0, \dots, v_\ell$ with $(v_{i-1}, v_i) \in E$ for all $i \in [\ell]$ and with $v_0 = v_\ell$, then the element $s = a_{(v_0, v_1)} \cdots a_{(v_{\ell-1}, v_\ell)}$ is the identity on $v_0$ and maps all other points to $0$. Therefore, $s$ is idempotent. Note that $a_{(v_0, v_1)} s$ is the zero element but $s a_{(v_0, v_1)}$ is not which means that $s$ is not central. Moreover, the sequence of generators defining $s$ forms a cycle at $v_0$ that is longer than length 1. Thus, the semigroup is not $\gR$-trivial by Lemma~\ref{lem:Rtrivial}.
\end{proof}

\begin{theorem}\label{thm:Rtrivial}
  $\Rtrivial$ is $\NL$-complete.
\end{theorem}
\begin{proof}
  The graph $\Gamma(\{a_1,\dots,a_k\},[n])$ can be computed by a log-space transducer.
  In another log-space transduction, we remove all cycles of length 1.
  We then use the $\NL$ algorithm for graph acyclicity. So, $\Rtrivial$ is in $\NL$ by Lemma~\ref{lem:Rtrivial}.
  For $\NL$-hardness, note that nilpotent semigroups are $\gR$-trivial and thus Lemma~\ref{lem:Rtrivial-hard} reduces the $\NL$-complete problem of directed graph acyclicity to $\Rtrivial$.
\end{proof}

\section{Semigroup Equations} \label{EquationSection}

Let $X = \os{x_1,\dots,x_m}$ be a nonempty set of variables. Let $X^+$ denote the set of all nonempty words over $X$, and let $u$ and $v$ be two fixed nonempty words over $X$.
We say that a semigroup \emph{$S$ models $u \eq v$} if $h(u) = h(v)$ holds for all homomorphisms $h \colon X^+ \to S$.
For a fixed identity $u \eq v$, define the following problem:

\medskip
{$\Modelbf(u \eq v)$}
\begin{itemize}
\item Input: $a_1,\dots,a_k \in T_n$
\item Problem: Does $\langle a_1,\dots,a_k \rangle$ model $u \eq v$?
\end{itemize}

This problem is the dual of the well-known identity checking problem for a fixed semigroup $S$: given words $u,v$ as input, decide whether $S$ models $u \eq v$. See \cite{AVG:CIC} for background and complexity results on identity checking; in particular, examples of semigroups for which it is $\co\NP$-complete.
In contrast, we will prove that $\Model(u \eq v)$ is in $\NL$ for any fixed identity $u \eq v$ and $\NL$-complete for specific identities. Before describing the algoirithm, we will consider two generalizations of this problem.

So called \emph{$\omega$-identities} are often used to define varieties of finite semigroups. In addition to concatenating variables, they also allow for using the $\omega$-operator to take idempotent powers.
However, testing whether a semigroup models a fixed $\omega$-identity is $\PSPACE$-complete for very simple identities already: the $\omega$-identity $x^\omega x = x^\omega$ defines the class of finite aperiodic semigroups and testing aperiodicity of transformation semigroups is known to be $\PSPACE$-complete~\cite{CH:FAA}.
Part of the hardness of this problem results from the fact that the same variable can appear both inside and outside the scope of an $\omega$-operator. We will prove that the complexity drops to $\NL$ if we disallow such occurrences of variables as well as nesting of $\omega$-operators.
This variant of $\omega$-identities will be called \emph{quasi-identities} and is formalized using slightly different notation below.

For containment in $\NL$, we will actually prove that the following generalization is in $\NL$.
A semigroup \emph{$S$ models $x_1 = x_1^2,\dots,x_s = x_s^2 \Rightarrow u = v$} if for all homomorphisms $h \colon X^+ \to S$ with $h(x_1), \dots, h(x_s)$ idempotent, we have $h(u) = h(v)$.

\medskip
{$\Modelbf(x_1 \eq x_1^2,\dots,x_s \eq x_s^2 \Rightarrow u \eq v)$}
\begin{itemize}
\item Input: $a_1,\dots,a_k \in T_n$
\item Problem: Does $\langle a_1,\dots,a_k \rangle$ model $x_1 \eq x_1^2,\dots,x_s \eq x_s^2 \Rightarrow u \eq v$?
\end{itemize}

\begin{theorem}
  Let $X = \os{x_1, \dots, x_m}$ be a nonempty finite set of variables. Let
  $e \in \os{0, \dots, m}$ and let $u, v \in X^+$.
  Then, $\Model(x_1 \eq x_1^2,\dots,x_e \eq x_e^2 \Rightarrow u \eq v)$ belongs to $\NL$.
  \label{thm:model}
\end{theorem}
\begin{proof}
  We describe an $\NL$ algorithm to test whether a semigroup $S = \langle a_1, \dots, a_k \rangle$ does \emph{not} satisfy a fixed quasi-identity
  \begin{align*}
    x_1 \eq x_1^2,\dots,x_e \eq x_e^2 \Rightarrow x_{i_1} \cdots x_{i_\ell} = x_{j_1} \cdots x_{j_r}
  \end{align*}
  where $X = \os{x_1, \dots, x_m}$ is a fixed set of variables, $i_1, \dots, i_\ell, j_1, \dots, j_r \in [m]$, and $0 \le e \le m$.
  Since $\NL$ is closed under complementation, this implies that the decision problem $\Model(x_1 \eq x_1^2,\dots,x_e \eq x_e^2 \Rightarrow u \eq v)$ belongs to $\NL$.
  For each $i \in [m]$, we let $P_1(i) = \set{p \in [\ell]}{i_p = i}$ and $P_2(i) = \set{p \in [r]}{j_p = i}$.
  The algorithm is depicted in Algorithm~\ref{alg:models}.
  Since $\ell + r$ is a constant, the algorithm only requires logarithmic
  space.

  \begin{algorithm}
  \caption{$\co\NL$ algorithm for $\Model(x_1 \eq x_1^2,\dots,x_e \eq x_e^2 \Rightarrow u \eq v)$}
    \label{alg:models}
    \begin{algorithmic}[1]
      \Input{$a_1,\dotsc,a_k \in T_n$}
      \Output{Does $\langle a_1,\dots,a_k \rangle$ \emph{not} model $x_1 \eq x_1^2,\dots,x_e \eq x_e^2 \Rightarrow u \eq v$?}
      \State guess integers $p_1, \dots, p_{\ell+1}, q_1, \dots, q_{r+1} \in [n]$
      \LineIf{$p_1 \ne q_1$ or $p_{\ell+1} = q_{r+1}$}{\Reject}
      \ForAll{$i \in [m]$}
      \LineForAll{$j \in [\ell]$}{$p_j' := p_j$,\, $p_j'' := p_{j+1}$}
      \LineForAll{$j \in [r]$}{$q_j' := q_j$,\, $q_j'' := q_{j+1}$}
        \Repeat
          \State guess $c \in [k]$
          \LineForAll{$j \in P_1(i)$}{$p_j' := p_j'a_c$,\, $p_j'' := p_j'' a_c$}
          \LineForAll{$j \in P_2(i)$}{$q_j' := q_j'a_c$,\, $q_j'' := q_j'' a_c$}
          \Until{$\forall j \in P_1(i) \colon p_j' = p_{j+1}$ and $\forall j \in P_2(i) \colon q_j' = q_{j+1}$ and \\ \hspace{1.35cm} $i \in [e] \Rightarrow (\forall j \in P_1(i) \colon p_j'' = p_{j+1}$ and $\forall j \in P_2(i) \colon q_j'' = q_{j+1})$}
      \EndFor
      \State\Accept
    \end{algorithmic}
  \end{algorithm}

  The process corresponds to nondeterministically replacing each variable
  $x_i$ by an element of $S$ such that the left-hand side and the right-hand
  side of the equation map the point $p_1 = q_1 \in [n]$ to distinct points $p_{\ell+1}, q_{r+1} \in [n]$.
  The variables $p_i''$ and $q_i''$ ensure that the substitutions can be
  performed in a way such that all variables $x_i$ with $i \le e$ are
  substituted by idempotent elements. A formal correctness proof follows.
  
  First, suppose that the input $S = \langle a_1,\dots,a_k\rangle$ does not model the quasi-identity $x_1 = x_1^2,\dots,x_e = x_e^2 \rightarrow x_{i_1}\cdots x_{i_\ell} = x_{j_1}\cdots x_{j_r}$. This means that there are elements $s_1,\dots,s_m \in S$ such that $s_1 = s_1^2,\dots,s_e = s_e^2$ and $s_{i_1}\cdots s_{i_\ell} \neq s_{j_1}\cdots s_{j_r}$.
  Pick a $p_1 \in [n]$ such that $p_1 s_{i_1} \cdots s_{i_\ell} \neq p_1 s_{j_1} \cdots s_{j_r}$. Let $q_1 := p_1$. For each $f \in [\ell]$, let $p_f = p_1 s_{i_1} \cdots s_{i_{f-1}}$. For each $f \in [r]$, let $q_f = q_1 s_{j_1} \cdots s_{j_{f-1}}$.

  To verify that the algorithm will accept the input, consider any $s_i \in \{s_1,\dots,s_m\}$. Let $s_i = a_{c_1} \cdots a_{c_g}$ with $c_1, \dots, c_g \in [k]$. Lines 6--10 will successively guess the generators and transform $p_j'$ for each $j \in P_1(i)$; likewise for $q_i'$. When this loop completes, the algorithm will have transformed each $p_j'$ and $q_j'$ to $p_js_i$ and $q_js_i$, respectively. Consequently, $p_j' = p_j s_i = p_{j+1}$ for each $j \in P_1(i)$ and $q'_j = q_js_i = q_{j+1}$ for each $j \in P_2(i)$. 
  Furthermore, this loop transforms $p_j''$ and $q_j''$, which tracks how the generators act on $p_{j+1}$ and $q_{j+1}$. By the end, $p_j'' = p_{j+1}s_i$ and $q_j'' = q_{j+1}s_i$. For each $j \in P_1(i) \cap [e]$, $s_i$ will be idempotent so that $p_j'' = p_{j+1}s_i = (p_j s_i) s_i = p_j s_i = p_{j+1}$. Likewise, $q_j'' = q_{j+1}$ for each $j \in P_2(i) \cap [e]$.
  
  We now prove that if the algorithm accepts, then $S = \langle a_1,\dots,a_k\rangle$ does not model $x_1 = x_1^2,\dots,x_e = x_e^2 \Rightarrow u = v$.
  Let $p_1,\dots,p_{\ell+1},q_1,\dots,q_{r+1}$ be the guessed integers in Line 1. For each $i \in [m]$, let $s_i = a_{c_1}\cdots a_{c_g}$ be the sequence of guessed generators in Line 7.
  Then for each $j \in P_1(i)$, $p_j s_i = p_{j+1}$ and for each $j \in P_2(i)$, $q_j s_i = q_{j+1}$. Let $s_i^\omega$ be the idempotent power of $s_i$. Then for each $j \in P_1(i)$ with $j \le e$, we have $p_{j+1}s_i^\omega = p_{j+1} s_i = p_{j+1}$ and for each $j \in P_2(i)$ with $j \le e$, we have $q_{j+1}s_i^\omega = q_{j+1} s_i = q_{j+1}$.
  
  By the definitions of $P_1(i)$ and $P_2(i)$, this demonstrates that $p_1 h(u)
  = p_{\ell+1}$ and $q_1 h(v) = q_{r+1}$ where $h \colon X^+ \to S$ is the
  homomorphism defined by $h(x_i) = s_i^\omega$ for all $i \in [e]$ and $h(x_i)
  = s_i$ for all $i \in \os{e+1,\dots,m}$.
  By Line~2 of the algorithm, we obtain $h(u) \ne h(v)$, thereby concluding the proof.
\end{proof}

Theorem \ref{thm:model} immediately yields the following.

\begin{corollary} \label{IdemCor}
Given generators $a_1,\dots,a_k \in T_n$, there are $\NL$ algorithms to check the following properties of $S = \langle a_1,\dots,a_k\rangle$:
\begin{enumerate}
\item all elements in $S$ are idempotent;
\item all idempotents are central in $S$;
\item all idempotents commute;
\item the product of any two idempotents is idempotent.
\end{enumerate}
\end{corollary}

We now give lower bounds for $\Model(x_1 \eq x_1^2,\dots,x_s \eq x_s^2 \Rightarrow u \eq v)$. We
certainly cannot state such a result without restricting the class of considered
equations.
For example, $\Model(x \eq x)$ is certainly not complete for any meaningful
complexity class.
A transformation semigroup is commutative if and only if
all generators commute, which implies $\Model(xy \eq yx) \in \AC^0$.
Another interesting setting is the class of finite groups which is defined by the $\omega$-identities $x^\omega y = y = y x^\omega$. While the identities look similar to the identity for aperiodicity, no variables appear both inside and outside the scope of an $\omega$-operator, so the problem is in $\NL$ by Theorem~\ref{thm:model}. And using a different approach mentioned in Section~\ref{sec:intro}, it is easy to see that deciding whether a transformation semigroup is a group is actually in $\AC^0$.

At the same time, $\Model(u = v)$ is $\NL$-complete for some equations.
The following two theorems are direct consequences of Lemma~\ref{lem:nil-hard} and Lemma~\ref{lem:Rtrivial-hard}.
\begin{theorem} \label{SimpleThm}
  $\Model(x^2y = x^2)$ is $\NL$-complete.
\end{theorem}
\begin{proof}
Given a DFA $\autA$, Lemma~\ref{lem:nil-hard} yields a transformation semigroup $S$ that reduces DFAEmptiness to $\Model(x^2y=x^2)$. If $L(\autA) = \emptyset$, then $S$ satisfies $x^2y = x^2$. If $L(\autA) \neq \emptyset$, then there is an idempotent $e \in S$ and some $s \in S$ such that $es \neq e$. Then, $e^2s = es \neq e = e^2$.
\end{proof}

\begin{theorem} \label{CentralIdempotentsThm}
  Testing whether all idempotents are central is $\NL$-complete.
\end{theorem}
\begin{proof}
Given a directed graph $G$, Lemma~\ref{lem:Rtrivial-hard} yields a transformation semigroup $S$ that reduces the directed acyclic graph problem to testing whether all idempotents are central. If $G$ is not acyclic, $S$ has a non-central idempotent. If $G$ is acyclic, $S$ is nilpotent and thus has only one idempotent; its zero, which is central.
\end{proof}

This suggests that a classification of the exact complexity of testing whether a given semigroup satisfies a fixed $\omega$-identity (in terms of the structure of the considered identity) is hard.

We return to a problem mentioned at the end of Section~\ref{GlobalSection}: determining if a commutative semigroup is regular. We show this is in NL by proving a harder result; determining whether a semigroup is completely regular.

\medskip
$\CompRegbf$
\begin{itemize}
\item Input: $a_1,\dots,a_k \in T_n$
\item Problem: Is $\langle a_1,\dots,a_k \rangle$ completely regular?
\end{itemize}
\begin{lemma} \label{comregLemma}
A transformation $s \in T_n$ generates a subgroup iff $s|_{[n]s}$ is a permutation.
\end{lemma}
\begin{proof}
If $s$ permutes $[n]s$, then $\langle s \rangle$ is isomorphic to a subgroup of the symmetric group over ${|[n]s|}$ elements.

Conversely, let $s \in T_n$ generate a subgroup. Then $s = s^k$ for some natural number $k>1$. Pick any $p,q \in [n]$ such that $ps \neq qs$. Because $ps^k \neq qs^k$, then $ps^2 \neq qs^2$. Thus, $ps \neq qs$ implies $ps^2 \neq qs^2$. That is, $s$ permutes $[n]s$. 
\end{proof}
\begin{theorem} \label{ComRegThm}
$\CompReg$ is in $\NL$.
\end{theorem}
\begin{proof}
Let $S := \langle a_1,\dots,a_k\rangle \leq T_n$. By Lemma \ref{comregLemma}, this semigroup is completely regular if and only if the restriction of each $s \in S$ to its image is a permutation.
Since $\NL$ is closed under complementation, we only need to give an $\NL$ algorithm for the complement of $\CompReg$. Thus, it suffices to verify that there exists some element of $S$ that is not a permutation on its image.
\end{proof}

\begin{corollary}
Deciding whether a commutative semigroup given by generators is regular is in $\NL$.
\end{corollary}
\begin{proof}
A commutative semigroup is regular iff it is completely regular by \cite[p.107]{HO:FST}. So, this follows directly from Theorem \ref{ComRegThm}.
\end{proof}

Thus, Theorems \ref{ComRegThm} and \ref{IdemCor} yield:

\begin{corollary} \label{CliffCor}
Given generators $a_1,\dots,a_k \in T_n$, there is an $\NL$ algorithm to check whether $\langle a_1,\dots,a_k\rangle$ is a Clifford semigroup.
\end{corollary}

\section{Polynomial Time Problems} \label{PolynomialSection}

We will prove that the following problem can be executed in polynomial time.

\medskip
$\LeftIdbf$
\begin{itemize}
\item Input: $a_1,\dots,a_k \in T_n$
\item Problem: Enumerate the left identities of $\langle a_1,\dots,a_k \rangle$.
\end{itemize}

\begin{lemma} \label{leftidlem}
Let $k,n \in \mathbb{N}$, let $a_1,\dots,a_k \in T_n$, and $S := \langle a_1,\dots,a_k\rangle$. Then an element $\ell \in S$ is a left identity of $S$ iff there is an $i \in [k]$ such that $\overline{a_i}$ permutes $[n] / \ker(S)$ and $\ell = a_i^\omega$.
\end{lemma}
\begin{proof}
For the backward direction, let $\overline{a_i}$ permute $[n] / \ker(S)$. Then $\overline{a_i}^\omega$ is the identity map on $[n]/\ker(S)$. That is, for any $q \in [n]$, we have that $\llbracket q a_i^\omega \rrbracket = \llbracket q \rrbracket \overline{a_i}^\omega = \llbracket q \rrbracket$. Thus, $qa_i^\omega s = qs$ for every $s \in S$, making $a_i^\omega$ a left identity of $S$.

For the forward direction, let $\ell$ be a left identity of $S$. Then $q\ell s = qs$ for every $q \in [n]$ and every $s \in S$, meaning $\llbracket q\ell \rrbracket = \llbracket q \rrbracket$.
Thus, $\overline{\ell}$ is the identity map on $[n]/\ker(S)$. Consequently, $\ell=sa_i$ for some generator $a_i$ where $\overline{a_i}$ is a permutation of $[n] / \ker(S)$. As proved in the paragraph above, $\overline{a_i}^\omega$ is also the identity map on $[n]/\ker(S)$ and thus $\overline{s a_i} = \overline{a_i}^\omega$. Then the following holds for every $q \in [n]$:
\[\llbracket qs \rrbracket \overline{a_i} = \llbracket q \rrbracket \overline{sa_i} = \llbracket q \rrbracket \overline{a_i}^\omega = \llbracket qa_i^{\omega-1} \rrbracket \overline{a_i}.\]
Since $\overline{a_i}$ is a permutation of $[n] /\ker(S)$, then $\llbracket qs \rrbracket = \llbracket qa_i^{\omega-1} \rrbracket$ and thus $qst = qa_i^{\omega-1}t$ for every $t \in S$. In particular, $qsa_i = qa_i^\omega$, proving $\ell = a_i^\omega$.
\end{proof}
\begin{theorem} \label{LidsThm}
$\LeftId$ can be computed in polynomial time.
\end{theorem}
\begin{proof}
By Lemma \ref{leftidlem}, we need only produce the idempotent powers of generators that induce permutations of $[n] / \ker(S)$. Certainly, we can produce $[n] /\ker(S)$ in polynomial time. Let $a_i$ be a generator such that $\overline{a_i}$ is a permutation of $[n]/\ker(S)$. Represent $\overline{a_i}$ in cycle notation and let $m$ be the least common multiple of the lengths of the cycles. Then $\overline{a_i}^m$ is idempotent and we claim that $a_i^m$ is idempotent as well. Pick any $q \in [n]$. Because $\overline{a_i}^m$ is idempotent:
\[\llbracket q a_i^{m-1}\rrbracket \overline{a_i} = \llbracket q \rrbracket \overline{a_i}^m = \llbracket q \rrbracket \overline{a_i}^{2m} = \llbracket qa_i^{2m-1}\rrbracket \overline{a_i}. \]
Because $\overline{a_i}$ is a permutation, then $\llbracket q a_i^{m-1} \rrbracket = \llbracket q a_i^{2m-1}\rrbracket$. That is, $qa_i^{m-1}s = qa_i^{2m-1}s$ for every $s \in S$.
In particular, $qa_i^m = qa_i^{2m}$ and thus $a_i^m$ is idempotent. Because we can find the idempotent power of $a_i$ in polynomial time, we can compute all left identities in polynomial time.
\end{proof}

We will prove that the following problem can be executed in polynomial time.

\medskip
$\RightIdbf$
\begin{itemize}
\item Input: $a_1,\dots,a_k \in T_n$
\item Problem: Enumerate the right identities of $\langle a_1,\dots,a_k\rangle$.
\end{itemize}

Note that $r \in S$ is a right identity of $S$ iff $\widetilde{r}$ is the identity of~$\widetilde{S}$, since for any $q \in [n]$ and any $s \in S$, $qsr = qs$ iff $qs\widetilde{r} = qs$.

\begin{lemma} \label{rightidlem}
Let $k,n \in \mathbb{N}$, let $a_1,\dots,a_k \in T_n$, and $S := \langle a_1,\dots,a_k\rangle$. Then an element $r \in S$ is a right identity of $S$ iff there is an $i \in [k]$ such that $\widetilde{a_i}$ permutes $[n]S$ and $r$ equals the idempotent power of $a_i$.
\end{lemma}
\begin{proof}
For the backward direction, let $\widetilde{a_i}$ permute $[n]S$. Then $\widetilde{a_i}^\omega$ is the identity map of $\widetilde{S}$. Then, $qsa_i^\omega = qs$ for every $q \in [n]$ and every $s \in S$, making $a_i^\omega$ a right identity of $S$.

For the forward direction, let $r$ be a right identity of $S$. Then $qsr = qs$ for every $q \in [n]$ and every $s \in S$, making $\widetilde{r}$ the identity map on $[n]S$. Then, $r = a_ib$ for some generator $a_i$ where $\widetilde{a_i}$ is a permutation of $[n]S$. As proved in the paragraph above, $\widetilde{a_i}^\omega$ is also the identity map of $[n]S$ and thus $\widetilde{r} = \widetilde{a_i}^\omega$. Then the following holds for every $q \in [n]$:
\[ qa_ib = qa_i\widetilde{a_i}^\omega b = qa_i^\omega\widetilde{a_ib} = qa_i^\omega \]
Therefore, $r$ equals the idempotent power of $a_i$.
\end{proof}
\begin{theorem} \label{RidsThm}
$\RightId$ can be computed in polynomial time.
\end{theorem}
\begin{proof}
By Lemma \ref{rightidlem}, we need only produce the idempotent powers of generators that induce  permutations of $[n]S$. Note that $[n]S = \bigcup_{i \in [k]} [n]a_i$ can be produced in polynomial time. Let $a_i$ be a generator such that $\widetilde{a_i}$ is a permutation of $[n]S$. Represent $\widetilde{a_i}$ in cycle notation and let $m$ be the least common multiple of the lengths of the cycles. Then $\widetilde{a_i}^m$ is idempotent and we claim that $a_i^m$ is idempotent as well. Because $\widetilde{a_i}^m$ is the identity map of $[n]S$, then for any $q \in [n]$, we have that $qa_i^m = qa_i^m\widetilde{a_i}^m = qa_i^{2m}$. Thus, $a_i^m$ is idempotent. Since we found $a_i$ and $m$ in polynomial time, then we can compute all right identities in polynomial time.
\end{proof}

\section{Global Properties} \label{GlobalSection}

In this section, we consider properties of a transformation semigroup $S$ that cannot be checked locally in the sense of Section \ref{LocalSection}.
We now define the following four problems:

\medskip
$\Regbf$
\begin{itemize}
\item Input: $a_1, \dots,a_k \in T_n$
\item Problem: Is there $s \in \langle a_1, \dots,a_{k} \rangle$ such that $a_k s a_k = a_k$?
\end{itemize}

\medskip
$\GenRegbf$
\begin{itemize}
\item Input: $a_1, \dots,a_k,s \in T_n$
\item Problem: Is there $t \in \langle a_1, \dots,a_{k} \rangle$ such that $sts = s$?
\end{itemize}

\medskip
$\GenWeakInvbf$
\begin{itemize}
\item Input: $a_1, \dots,a_k,s \in T_n$
\item Problem: Is there $t \in \langle a_1, \dots,a_{k} \rangle$ such that $tst = t$?
\end{itemize}

\medskip
$\GenInvbf$
\begin{itemize}
\item Input: $a_1, \dots,a_k,s \in T_n$
\item Problem: Is there $t \in \langle a_1, \dots,a_{k} \rangle$ that is an inverse of $s$?
\end{itemize}

We now prove that all four of these problems are $\PSPACE$-complete. We begin by showing that $\Reg$ and $\GenWeakInv$ are $\PSPACE$-hard by adapting a proof by Christian Brandl and Hans Ulrich Simon \cite{BS:CATM}. We reduce the following problem known to be $\PSPACE$-complete \cite{KO:LBN}.

\medskip
{$\DFAIsectbf$}
\begin{itemize}
\item Input: Deterministic finite automata (DFA) $\autA_1,\dots,\autA_k$ over a shared alphabet $\Sigma$, each with a unique final state.
\item Problem: Is there $w \in \Sigma^*$ that sends each initial state to its corresponding final state?
\end{itemize}

\begin{lemma}\label{RegEl} $\Reg$ and $\GenWeakInv$ are $\PSPACE$-hard. \end{lemma}
\begin{proof}
  Let $\autA_1, \dots,\autA_k$ be DFAs with corresponding sets of states $Q_1,\dots,Q_k$, initial states $p_1,\dots,p_k$, final states $q_1,\dots,q_k$, and a common alphabet $\Sigma = \{a_1,\dots,a_m\}$.
  We assume without loss of generality that the sets of states are pairwise disjoint, \ie $Q_i \cap Q_j = \emptyset$ for $1 \le i < j \le k$.
  Let
  \begin{align*}
    Q := \{0\} \cup \bigcup \limits_{i=1}^k Q_i
  \end{align*}
  be the disjoint union of the DFA states along with a new state, denoted by $0$.
  Extend the action of $\Sigma$ to transformations on $Q$ by defining $0a_i = 0$ for each $i \in \os{1, \dots, m}$. Define the following additional transformation:
$$qb := \begin{cases}
 p_j & \text{if }q = q_j,\\ 
 0 & \text{otherwise}.
\end{cases} $$

We claim there is a word $w$ in the intersection of the DFA languages iff $b$ is regular in $S := \langle a_1,\dots,a_m,b \rangle$. Assume there is a word $w$ accepted by each $\autA_1,\dots,\autA_k$ so that $p_jw = q_j$ for each $j \in [k]$. Then we can verify that $qb = qbwb$ for all $q \in Q$. If $q$ is not a final state, then $qbwb = 0wb = 0b = 0 = qb$. If $q = q_j$ for some $j \in [k]$, then $q_jbwb = p_jwb = q_jb$. 

Now let $b = bsb$ for some transformation $s \in S$. If $q = q_j$, then $q_jb = p_j$. So, $p_j = p_jsb$. Because $b$ sends everything to $0$ except for $q_j$ and since $0t = 0$ for any $t \in S$, then $s$ must contain a word $w$ that sends $p_j$ to $q_j$. This is true for each $j \in [k]$ so that $w$ must be in each of the DFA languages. Thus, the reduction is complete and $\Reg$ is $\PSPACE$-hard. 

We now reduce $\DFAIsect$ to $\GenWeakInv$. Define:

$$qc := \begin{cases}
s_j & \text{if }q \in Q_j,\\
0 & \text{if }q = 0.
\end{cases}$$

We claim there is a word $w$ accepted by each $\autA_1,\dots,\autA_k$ iff $b$ has a weak inverse in $S := \langle a_1,\dots,a_k,c \rangle$. Assume $w$ is accepted by each $\autA_1,\dots,\autA_k$. We claim $cw$ is a weak inverse for $b$; that is, $qcwbcw = qcw$ for each $q \in Q$. If $q \in Q_j$, then $q cwbcw = p_jwbcw = q_jbcw = p_jcw = p_jw = qcw$. And clearly $0cwbcw = 0 = 0cw$. Conversely, assume there is a $t \in S$ satisfying $tbt = t$. Note that $p_jt \in Q_j$ for all $j \in [k]$. If $p_jt \neq q_j$, then $p_jtbt = 0t = 0 \not \in Q_j$. Thus, $p_jt = q_j$. Because $c$ can only reset states in $Q_j$ back to $p_j$, then $t$ must end with a word $w$ that sends $p_j$ to $q_j$. Then $w$ is accepted by $\autA_1,\dots,\autA_k$. Thus, $\GenWeakInv$ is $\PSPACE$-hard.
\end{proof}

\begin{lemma}\label{GenInv} $\Reg$, $\GenReg$, $\GenWeakInv$, and $\GenInv$ are in $\PSPACE$. \end{lemma}
\begin{proof}
The non-deterministic Algorithm~\ref{alg:gi} correctly decides $\GenInv$. Since it only requires space to store a single transformation $c$ in $T_n$ and compute products $ca_i, bcb, cbc$, Algorithm~\ref{alg:gi} requires working space $O(n\log(n))$. Hence, $\GenInv$ is in $\NPSPACE$, which is $\PSPACE$ by Savitch's Theorem.

\begin{algorithm}
\caption{  \newline{} Function $\GenInvbf(a_1,\dotsc,a_k,b)$}
  \label{alg:gi}
  \begin{algorithmic}[1]
    \Input{$a_1,\dotsc,a_k,s\in T_n$}
    \Output{Is there $t\in\langle a_1,\dotsc,a_k \rangle$ such that $sts=s$ and $tst=t$?}
    \State guess $i \in [k]$ and let $t := a_i$
      \While{$sts\neq s$ or $tst\neq t$}
      \State guess $i \in [k]$ and let $t := ta_i$
      \EndWhile \\
      \Accept
  \end{algorithmic}
\end{algorithm}

Straighforward adaptations of line 2 in Algorithm~\ref{alg:gi} yield non-deterministic polynomial space algorithms for the other three decision problems. Hence these problems are in $\PSPACE$ as well.
\end{proof}

\begin{theorem} \label{RegThm}
$\Reg$, $\GenReg$, $\GenWeakInv$, and $\GenInv$ are all $\PSPACE$-complete. \end{theorem}
\begin{proof}
Note that every regular element in a semigroup has an inverse in that semigroup \cite[p.51]{HO:FST}. So, $\Reg$ can be reduced to $\GenReg$ and $\GenInv$ by simply letting $s=a_k$. Because $\Reg$ is $\PSPACE$-hard by Lemma \ref{RegEl}, then the other two are as well. And because $\GenReg$ and $\GenInv$ are in $\PSPACE$ by Lemma \ref{GenInv}, then all three problems are $\PSPACE$-complete. $\GenWeakInv$ is $\PSPACE$-complete since it is $\PSPACE$-hard by Lemma \ref{RegEl} and it is in $\PSPACE$ by Lemma \ref{GenInv}.
\end{proof}

Note that we did not define the problem of determining if an element $s \in S$ has a weak inverse in $S$, because this is always true for any element of a finite semigroup: we have $s^{\omega-1} s s^{\omega-1} = s^{2\omega-1} = s^{\omega-1}$, so $s^{\omega-1}$ is a weak inverse of $s$.

We can show that the complexity of determining if a semigroup given by generators is regular is in $\PSPACE$.  Note that we can nondeterministically guess an element of $S$ and, by Theorem \ref{RegThm}, verify in polynomial time that it is not regular. Thus, this problem is in $\co\NPSPACE$, which, by Savitch's Theorem, is $\PSPACE$. It is unknown whether checking semigroup regularity is $\PSPACE$-hard.

\section{Open Problems} \label{Open Problems}

The problems of testing whether a transformation semigroup is a band, has commuting idempotents, is orthodox, is completely regular or a Clifford semigroup were shown to be decidable in non-deterministic logarithmic space.
The problems of testing whether a transformation semigroup is regular or inverse were shown to be decidable in polynomial space.
Lower bounds for each of these problems are still open.

\section*{Acknowledgments}
We would like to thank Peter Mayr and Nik Ru\v{s}kuc for discussions on the material in this paper.

\end{document}